\documentclass[12pt]{amsart}

\usepackage{amsthm,amssymb,enumitem,mathtools}

\usepackage{tikz}

\usepackage{scalerel}

\usetikzlibrary{arrows,shapes,automata,backgrounds,decorations,petri,positioning}
\usetikzlibrary{decorations.pathreplacing,angles,quotes}

\tikzstyle{edge} = [fill,opacity=.5,fill opacity=.5,line cap=round, line join=round, line width=50pt]

\pgfdeclarelayer{background}
\pgfsetlayers{background,main}

\usepackage[margin=1in,letterpaper,portrait]{geometry}

\usepackage{amsthm}

\usepackage[latin1]{inputenc}
\usepackage{url}
\usepackage{subfigure}
\usepackage{color}
\usepackage{amsmath}
\usepackage{amsthm}
\usepackage{amstext}
\usepackage{amssymb}
\usepackage{amsfonts}
\usepackage{graphicx}
\usepackage{young}
\usepackage{multicol}
\usepackage{mathrsfs}
\usepackage[all]{xy}
\usepackage{tikz}
\usepackage{mathtools}
\usepackage{tabularx}
\usepackage{array}
\usepackage{commath}
\usepackage{ytableau}

\theoremstyle{plain}
\theoremstyle{definition}
\newtheorem{theorem}{Theorem}

\newtheorem{proposition}[theorem]{Proposition}
\newtheorem{corollary}[theorem]{Corollary}

\DeclareMathAlphabet{\mathpzc}{OT1}{pzc}{m}{it}

\newcommand{\total}[1]{\|#1\|}

\newcommand{\omitt}[1]{}

\usepackage{amssymb,amsmath,tabularx,graphicx}
\usepackage{tikz}
\usetikzlibrary[calc,intersections,through,backgrounds,arrows,decorations.pathmorphing]

\begin{document}

\title{Broken bricks and the pick-up sticks problem}

\author{T.~Kyle Petersen}
\address{Department of Mathematical Sciences, DePaul University, Chicago, IL, USA}
\email{t.kyle.petersen@depaul.edu}
\thanks{Research partially supported by Simons Foundation Collaboration Grant for Mathematicians 353772}%

\author{Bridget Eileen Tenner}
\address{Department of Mathematical Sciences, DePaul University, Chicago, IL, USA}
\email{bridget@math.depaul.edu}
\thanks{Research partially supported by Simons Foundation Collaboration Grant for Mathematicians 277603.}



\begin{abstract}
We generalize the well-known broken stick problem in several ways, including a discrete ``brick'' analogue and a sequential ``pick-up sticks/bricks'' version. The limit behavior of the broken brick problem gives a combinatorial proof of the broken stick problem. The pick-up version gives a variation on those scenarios, and we conclude by showing a greater context---namely, that the broken stick/brick problem and the pick-up sticks/bricks problem are two extremes in a family of interesting, and largely open, questions.
\end{abstract}

\maketitle

There is a classical probability exercise about forming a triangle from pieces of a stick.

\begin{quote}
\textbf{The broken stick problem -- classical version.} Consider a stick of fixed length. Pick two distinct interior points on the stick, independently and at random, and cut the stick at these two points. What is the probability that the resulting three pieces form a triangle?
\end{quote}

\noindent For example, if the stick has length $1$, then breaking the stick into segments of lengths $1/10$, $3/7$, and $1 - 1/10 - 3/7 = 33/70$ will produce a triangle, whereas breaking it into segments of lengths $1/10$, $3/8$, and $1 - 1/10 - 3/8 = 21/40$ will not produce a triangle, as shown in Figure~\ref{fig:broken stick examples}. 

\begin{figure}[htbp]
\begin{tikzpicture}[scale=10]
\draw[very thick] (0,0) -- (.471,0) -- (.052,.086) -- (0,0);
\end{tikzpicture}
\hspace{1in}
\begin{tikzpicture}[scale=10]
\draw[very thick] (0,0) -- (.525,0);
\draw[very thick] (0,0) -- ++(.098,.0196);
\draw[very thick] (.525,0) --++(-.374,.0196);
\end{tikzpicture}
\caption{Two breakings of a stick into three pieces, one of which can form a triangle and one of which cannot.}
\label{fig:broken stick examples}
\end{figure}
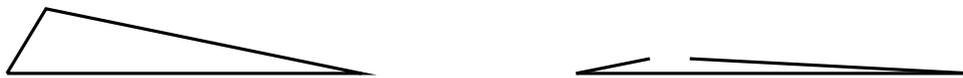

The classical broken stick problem can be answered by a nice argument in geometric probability, showing that we produce a triangle with probability $1/4$.

This problem generalizes naturally to arbitrary polygons, as follows.

\begin{quote}
\textbf{The broken stick problem -- general version.} Consider a stick of fixed length and a positive integer $k \ge 3$. Pick $k-1$ distinct interior points on the stick, independently and at random, and cut the stick at these $k-1$ points. What is the probability that the resulting $k$ pieces form a $k$-gon?
\end{quote}

The scenario of the broken stick problem has applications to a number of other fields, as discussed in \cite{uiuc report}. Another application is that, due to Proposition~\ref{prop:k-ilateral inequality} below, the general broken stick problem is related to a $k$-candidate plurality election in which no candidate wins a majority of the votes. Other generalizations and related discussions have appeared in \cite{crowdmath, gardner, ionascu pajitura, lemoine, poincare}.

The generalized broken stick problem has an elegant answer, as shown by D'Andrea and G\'omez.

\begin{theorem}[{cf.~\cite[Thm.~3]{dandrea gomez}}]\label{thm:splintered breaking}
Take a stick of fixed length and a positive integer $k \ge 3$. Pick $k-1$ distinct interior points on the stick, independently and at random, and cut the stick at these $k-1$ points. The probability that the resulting $k$ pieces form a $k$-gon is
$$1 - \frac{k}{2^{k-1}}.$$
\end{theorem}

The theorem can be proved geometrically, and we present that argument here as motivation. Suppose, without loss of generality, that the stick has unit length and consider, as the sample space, the interior of the unit simplex
\[
 \{ (x_1,\ldots,x_k) \in (0,1)^k : \sum x_i = 1 \}.
\]
We call this the ``sample space'' because we interpret a point $(x_1,\ldots,x_k)$ in this space as describing the stick having been cut at the points 
\[
0< x_1 < x_1+x_2<\cdots < x_1+x_2+\cdots+x_{k-1} < 1,
\]
to create segments of lengths $x_1$, $x_2$, \ldots, $x_{k-1}$, and $x_k = 1-(x_1 + \cdots + x_{k-1})$. It transpires (see Proposition \ref{prop:k-ilateral inequality} below) that the multiset
\[
 \{ x_1, x_2, \ldots, x_k \}
\]
of these lengths describes the side lengths of a $k$-gon, necessarily of unit perimeter, if and only if 
\begin{equation}\label{ineq:length at most half}
 x_i < \frac{1}{2} \quad \mbox{for all $i$.}
\end{equation}

\noindent The $x_i$ are positive and sum to $1$, so Inequality~\eqref{ineq:length at most half} can fail for at most one coordinate at a time. Saying, for example, that $x_k \geq 1/2$, is equivalent to requiring that the remaining coordinates satisfy the inequality  
\[
 x_1+x_2+\cdots+x_{k-1} \leq 1/2.
\]
This defines a subset of the sample space, which can be described as a contraction of the full simplex toward the corner $(0,\ldots,0,1)$:
$$(y_1,\ldots,y_k) \mapsto \left(\frac{y_1}{2}, \frac{y_2}{2}, \cdots, \frac{y_{k-1}}{2}, \frac{y_1+\cdots+y_{k-1}}{2}+y_k\right).$$
As such, the volume of this subset is $1/2^{k-1}$ of the volume of the full simplex. This argument holds for any of the $k$ coordinates failing Inequality~\eqref{ineq:length at most half}. Hence, the proportion of the sample space that has some coordinate failing that inequality is $k/2^{k-1}$, and so the desired probability is, indeed, $1 - k/2^{k-1}$. Figure \ref{fig:geom34} depicts the cases $k=3$ and $k=4$.

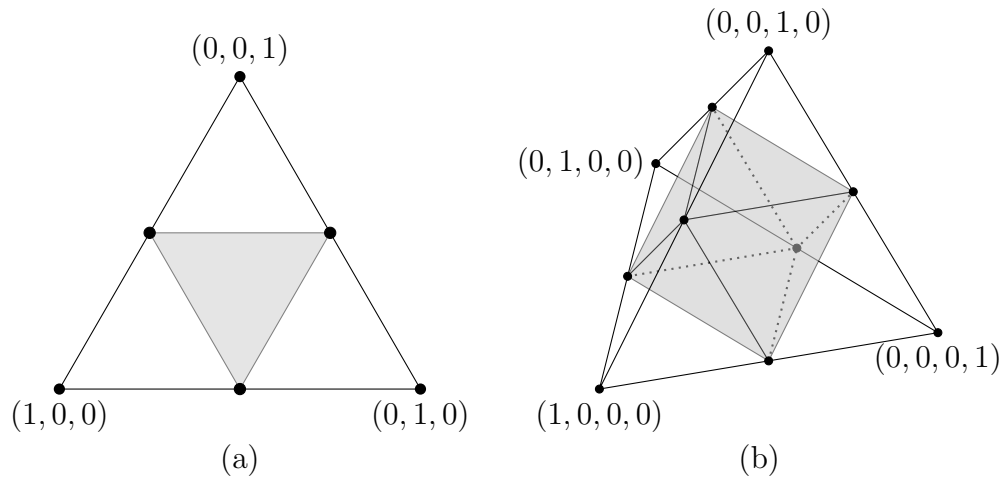
\begin{figure}
\begin{tabular}{c c}
\begin{tikzpicture}[cm={1,0,.5,.8660254,(0,0)},baseline=0,scale=.8]
\coordinate (a) at (0,0);
\coordinate (b) at (6,0);
\coordinate (c) at (0,6);
\coordinate (ac) at (0,3);
\coordinate (ab) at (3,0);
\coordinate (bc) at (3,3);
\draw (a)--(b)--(c)--(a);
\draw[fill=black] (a) node[fill=black,circle,inner sep=1.5] {} node[below] {$(1,0,0)$};
\draw[fill=black] (b) node[fill=black,circle,inner sep=1.5] {} node[below] {$(0,1,0)$};
\draw[fill=black] (c) node[fill=black,circle,inner sep=1.5] {} node[above] {$(0,0,1)$};
\draw[draw=black,fill=white!80!black,opacity=.5] (ab)--(bc)--(ac)--(ab);
\draw (ab) node[draw=black,fill=black,circle,inner sep=1.5] {};
\draw (ac) node[draw=black,fill=black,circle,inner sep=1.5] {};
\draw (bc) node[draw=black,fill=black,circle,inner sep=1.5] {};
\end{tikzpicture}
&
\begin{tikzpicture}[scale=.75,baseline=0]
\coordinate (a) at (0,0);
\coordinate (b) at (1,4);
\coordinate (c) at (3,6);
\coordinate (d) at (6,1);
\coordinate (ab) at (0.5,2);
\coordinate (ac) at (1.5,3);
\coordinate (ad) at (3,0.5);
\coordinate (bc) at (2,5);
\coordinate (bd) at (3.5,2.5);
\coordinate (cd) at (4.5,3.5);
\draw[fill=black] (a) circle (2pt) node[below] {$(1,0,0,0)$};
\draw[fill=black] (b) circle (2pt) node[left] {$(0,1,0,0)$};
\draw[fill=black] (c) circle (2pt) node[above] {$(0,0,1,0)$};
\draw[fill=black] (d) circle (2pt) node[below] {$(0,0,0,1)$};
\draw (b)--(d);
\draw[draw=none,fill=white!80!black,opacity=.2] (ab)--(bc)--(bd)--(ab);
\draw[draw=none,fill=white!80!black,opacity=.2] (ab)--(ad)--(bd)--(ab);
\draw[draw=none,fill=white!80!black,opacity=.2] (bd)--(bc)--(cd)--(bd);
\draw[draw=none,fill=white!80!black,opacity=.2] (bd)--(cd)--(ad)--(bd);
\draw[thick,dotted] (ab) -- (bd) -- (cd);
\draw[thick,dotted] (ad) -- (bd) -- (bc);
\draw[fill=black] (bd) circle (2pt);
\draw[draw=black,fill=white!80!black,opacity=.5] (ab)--(ac)--(bc)--(ab);
\draw[draw=black,fill=white!80!black,opacity=.5] (ab)--(ac)--(ad)--(ab);
\draw[draw=black,fill=white!80!black,opacity=.5] (ac)--(bc)--(cd)--(ac);
\draw[draw=black,fill=white!80!black,opacity=.5] (ad)--(ac)--(cd)--(ad);
\draw (a)--(b)--(c)--(d)--(a)--(c);
\draw[fill=black] (ab) circle (2pt);
\draw[fill=black] (ac) circle (2pt);
\draw[fill=black] (ad) circle (2pt);
\draw[fill=black] (bc) circle (2pt);
\draw[fill=black] (cd) circle (2pt);
\end{tikzpicture}\\
(a) & (b)
\end{tabular}
\caption{The geometric argument for $k=3$ and $k=4$. In (a), we see the probability of making a triangle is $1/4$ of the area of the sample space. In (b), we see the probability of making a quadrilateral is $1/2$ of the volume of the sample space.}\label{fig:geom34}
\end{figure}

The geometric probability argument for the broken stick problem is beautiful, but for two authors who spend most of their time counting things, a discrete version of the problem has great appeal. Thus we consider an analogue of the problem in which the stick has integer length and can only be broken at integer increments. Being sore-footed parents of young children, we think of this discrete version as a ``(LEGO) brick analogue.''

\begin{quote}
\textbf{The broken brick problem.} Let $n \ge k \ge 3$ be positive integers, and consider a stick of length $n$. Pick $k-1$ distinct interior integer points on the stick, independently and at random, and cut the stick at these $k-1$ points. What is the probability that the resulting $k$ pieces form a $k$-gon?
\end{quote}

The nice thing about this version of the problem is not only that it enables a combinatorial proof of D'Andrea and G\'omez's result, but also that it allows students and even small children to experiment with the question. For example with a stick of $n=10$ bricks, there are only $36$ ways to break the stick into $k = 3$ pieces, and the experimenter can record how many of these breakings result in a triangle. See Figure \ref{fig:kids}

\begin{figure}
\includegraphics[height=2.5in]{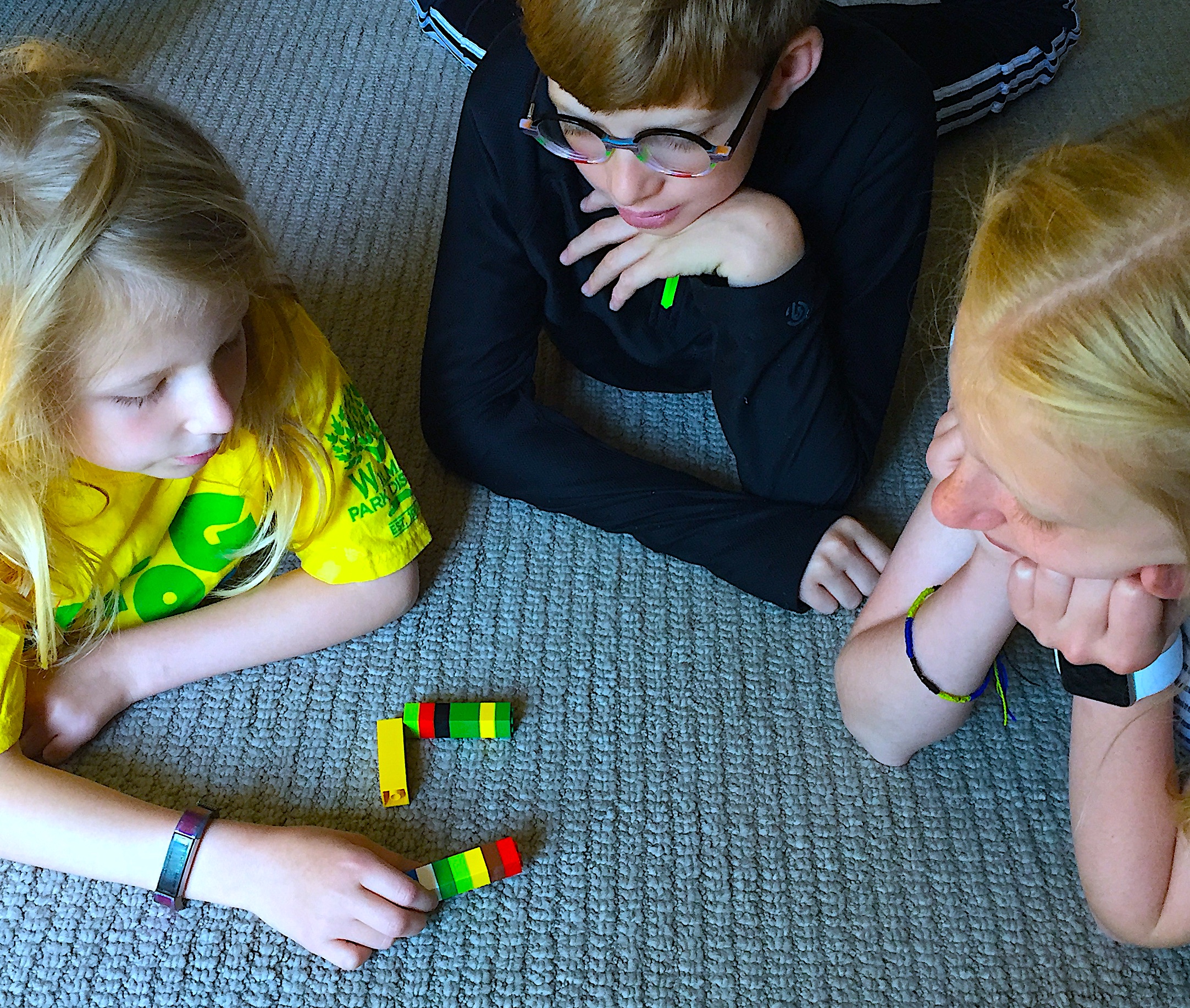} \hspace{.25in} \includegraphics[height=2.5in]{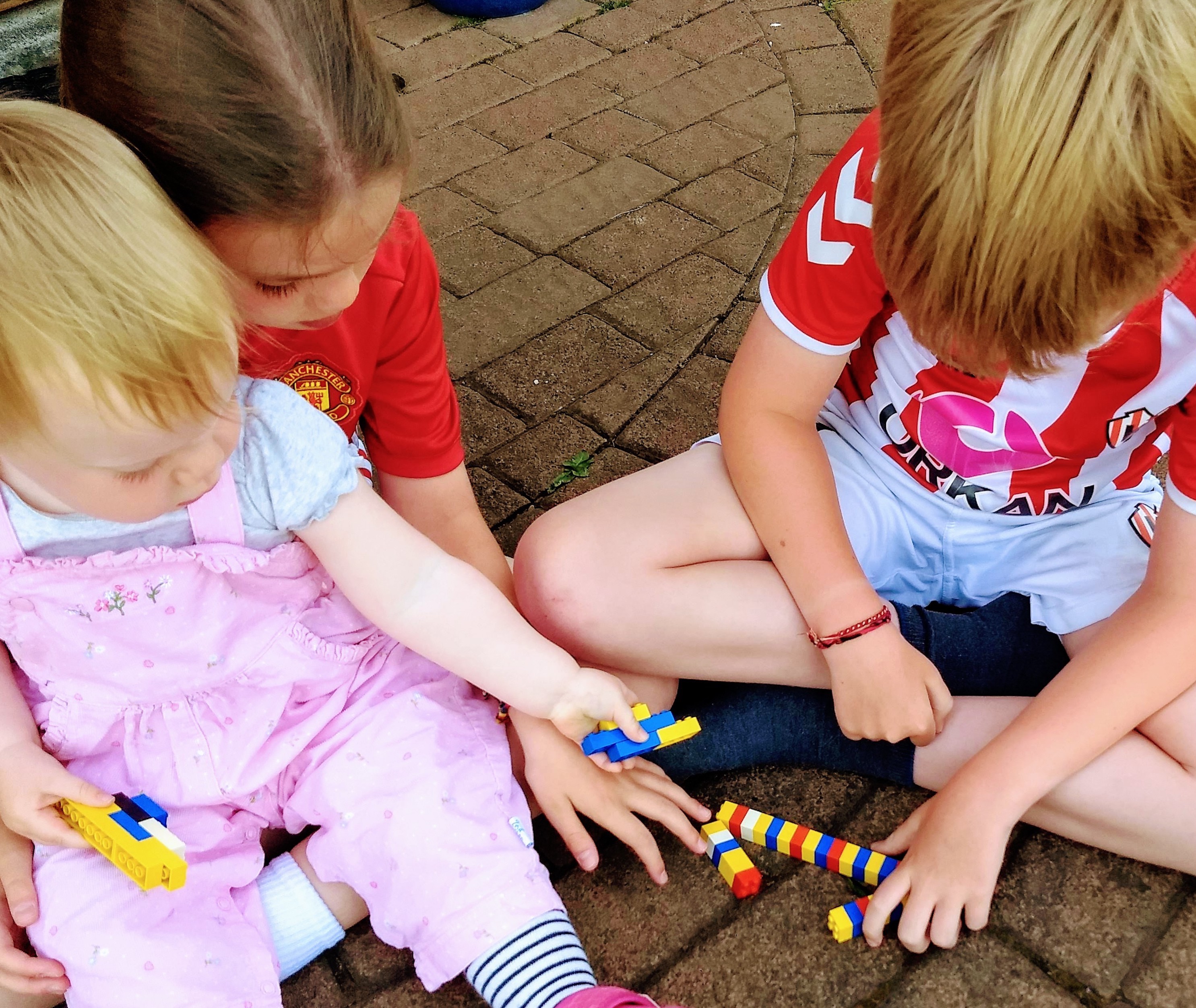}
\caption{Experimenting with sticks of LEGO bricks.}\label{fig:kids}
\end{figure}

\section{Polygonal inequalities}

The classical broken stick problem is a reference to the triangle inequality: a multiset $S$ of three positive numbers gives the side lengths of a triangle if and only if each (potential) side length is less than the sum of the other two (potential) side lengths. Thus the first step in solving generalizations of the broken stick problem is to find a $k$-gon analogue to the triangle inequality. More precisely, given a $k$-element multiset $S$ of positive numbers (later we will require that they be integers), is there a $k$-gon whose side lengths are the elements of $S$? What properties of $S$ must hold for such a $k$-gon to exist? The following result appears in \cite{dandrea gomez}, but we include a proof here in order to make their ``tweaking'' explicit. In what follows, we write $\total{S}$ to denote the sum of the elements of $S$.

\begin{proposition}[{cf.~\cite[Prop.~1]{dandrea gomez}}]\label{prop:k-ilateral inequality}
Fix a positive integer $k \ge 3$ and a $k$-element multiset $S$ of positive numbers. There exists a (convex) polygon whose side lengths are the elements of $S$ if and only if $x < \total{S} - x$, or equivalently,
\begin{equation}\label{eqn:side sum inequalities}
x < \frac{\total{S}}{2}
\end{equation}
for each $x \in S$.
\end{proposition}

\begin{proof}
We proceed by induction on $k$, noting that the case $k=3$ is precisely the triangle inequality. Assume, inductively, that the result holds for all $j$-gons, with $3 \le j < k$.

Suppose that a $k$-element multiset $S$ contains the side lengths of a $k$-gon $P$. Any diagonal of $P$ will separate the region into two polygons $Q_1$ and $Q_2$, each with fewer than $k$ sides, as shown in Figure~\ref{fig:drawing diagonal to prove inequality}. 
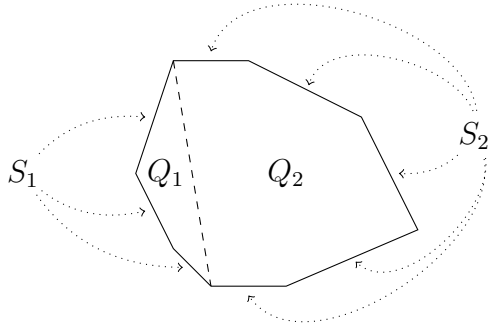
\begin{figure}[htbp]
\begin{tikzpicture}[scale=.5]
\draw (0,-1) -- (-1,0) -- (-2,2) -- (-1,5) -- (1,5) -- (4,3.5) -- (5.5,.5) -- (2,-1) -- (0,-1);
\draw[dashed] (-1,5) -- (0,-1);
\draw (-1.2,2) node {$Q_1$};
\draw (2,2) node {$Q_2$};
\path[dotted,->] (-5,2) edge[bend right] (-.75,-.5);
\path[dotted,->] (-5,2) edge[bend left] (-1.75,3.5);
\path[dotted,->] (-5,2) edge[bend right] (-1.75,1);
\fill[white] (-5,2) circle (.6);
\draw (-5,2) node {$S_1$};
\path[dotted,->] (7,3) edge[bend right,out=-90,in=-90] (0,5.25);
\path[dotted,->] (7,3) edge[bend right,out=-90,in=-90] (2.6,4.35);
\path[dotted,->] (7,3) edge[bend left] (5,2);
\path[dotted,->] (7,3) edge[bend left,out=90,in=90] (3.85,-.35);
\path[dotted,->] (7,3) edge[bend left,out=90,in=90] (1,-1.25);
\fill[white] (7,3) circle (.6);
\draw (7,3) node {$S_2$};
\end{tikzpicture}
\caption{Decomposing a $k$-gon by drawing a diagonal.}
\label{fig:drawing diagonal to prove inequality}
\end{figure}
Thus each $Q_i$ is subject to the inductive hypothesis, producing collections of inequalities as in Inequality~\eqref{eqn:side sum inequalities}. Let $d$ be the length of the diagonal separating $Q_1$ from $Q_2$, and decompose $S = S_1 \cup S_2$ so that the side lengths of $Q_1$ are $T_1 = S_1 \cup \{d\}$, and the side lengths of $Q_2$ are $T_2=S_2 \cup \{d\}$, as shown in Figure~\ref{fig:drawing diagonal to prove inequality}. Then by the inductive hypothesis,
$$x < \frac{\total{T_i}}{2}$$
for each $x \in T_i$ and $i\in \{1,2\}$. 

Let us see if the same holds with respect to $S$. Let $x \in S$. Then in particular, $x \in T_1 \cup T_2$ and $x \neq d$. Without loss of generality, suppose $x \in T_1$. First of all, 
$$x < \frac{\total{T_1}}{2} = \frac{\total{S_1}+d}{2}.$$
Likewise,
$$d < \frac{\total{T_2}}{2} = \frac{\total{S_2}+d}{2}.$$
Therefore $d/2 \le \total{S_2}/2$, and so
$$x < \frac{\total{T_1}}{2} = \frac{\total{S_1}+d}{2} < \frac{\total{S_1}+\total{S_2}}{2} = \frac{\total{S}}{2},$$
as desired.

Now suppose that $S = \{s_1,\ldots, s_k\}$ is a $k$-element multiset and that
\begin{equation}\label{eqn:inequality in proof}
x < \frac{\total{S}}{2}
\end{equation}
for all $x \in S$. We want to construct a $k$-gon whose side lengths are the elements of $S$. To do this, our goal is to find a value $d$ such that there is a triangle with sides of length $\{s_1,s_2,d\}$ and a $(k-1)$-gon with sides of length $\{s_3,s_4,\ldots, s_k,d\}$. Without loss of generality, suppose that $s_1 \ge s_2$, and that $s_k \ge s_i$ for all $i \in [3,k-1]$. Then, by the inductive hypothesis, this is equivalent to finding $d$ such that
$$s_1 - s_2 < d < s_1 + s_2$$
and
$$s_k - \big(s_3 + \cdots + s_{k-1}\big) < d < s_3 + \cdots + s_{k-1} + s_k.$$
The only way to have no such $d$ would be for these intervals to be disjoint, meaning that either 
$$s_1 + s_2 \le s_k - \big(s_3 + \cdots + s_{k-1}\big)$$
or
$$s_3 + \cdots + s_{k-1} + s_k \le s_1 - s_2.$$
However, the first of these would contradict Inequality~\eqref{eqn:inequality in proof} for $x = s_k$, and the second would contradict it for $x = s_1$. Thus the intervals must indeed overlap and so there must be such a length $d$. 

Then, by the inductive hypothesis, there is a triangle with sides of length $\{s_1,s_2,d\}$ and a $(k-1)$-gon with sides of length $\{s_3,s_4,\ldots,s_k,d\}$. Gluing these along their sides of length $d$ produces a $k$-gon with side lengths $\{s_1, \ldots, s_k\} = S$, as shown in Figure~\ref{fig:gluing together polygons}.
\begin{figure}[htbp]
\begin{tikzpicture}[scale=.5]
\draw (0,-1) -- (-1,0) -- (-2,2) -- (-1,5) -- (1,5) -- (4,3.5) -- (5.5,.5) -- (2,-1) -- (0,-1);
\draw[dashed] (-1,0) -- (2,-1);
\draw (.5,-.5) node[above right] {$d$};
\draw (1,-1) node[below] {$s_1$};
\draw (-.5,-.5) node[left] {$s_2$};
\draw (-1.5,1) node[left] {$s_3$};
\draw (-1.5,3.5) node[left] {$s_4$};
\draw (0,5) node[above] {$s_5$};
\draw (2.5,4.25) node[right] {$s_6$};
\draw (4.75,2) node[right] {$s_7$};
\draw (3.75,-.25) node[below] {$s_8$};
\end{tikzpicture}
\caption{Constructing a $k$-gon, given inequalities on side lengths.}
\label{fig:gluing together polygons}
\end{figure}
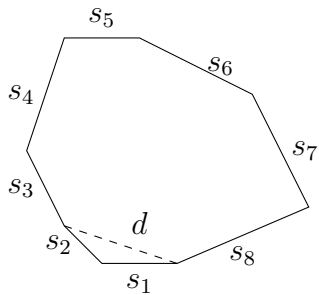
\end{proof}

There are some interesting things to note about the result and proof of Proposition~\ref{prop:k-ilateral inequality}. First, a given multiset $S$ need not produce a unique polygon. For example, the multiset $\{x,x,x,x\}$ describes the side lengths of infinitely many rhombi. Also, the construction in the proof of Proposition~\ref{prop:k-ilateral inequality} can be used to produce convex $k$-gons.

\section{Broken bricks}

Having characterized those multisets of positive numbers that can describe polygons, we now give our solution to the discrete analogue of the broken stick problem: the broken brick problem.

\begin{theorem}\label{thm:discrete breaking}
Let $n \ge k \ge 3$ be positive integers, and consider a stick of length $n$. Pick $k-1$ distinct interior integer points on the stick, independently and at random, and cut the stick at these $k-1$ points. The probability that the resulting $k$ pieces form a $k$-gon is
$$1 - \frac{k\binom{\lfloor n/2 \rfloor}{k-1}}{\binom{n-1}{k-1}}.$$
\end{theorem}

\begin{proof}
Fix values of $n$ and $k$ with $n\geq k \geq 3$. The sample space for the broken brick problem is the set of integer points
\[
 S = \{ (x_1,\ldots,x_k) : x_i \in \mathbb{Z},\ x_i \ge 1, \mbox{ and } \sum x_i = n \},
\]
with the same interpretation as before. We will call each element of $S$, which represents a collection of the segments obtained from the full stick, an \emph{inventory}.

In combinatorics, vectors satisfying the rules of membership in $S$ are called \emph{compositions} of $n$ into $k$ \emph{parts}. The number of compositions of $n$ into $k$ parts is $\binom{n-1}{k-1}$, as shown by the following argument. Each $x_i$ is positive, so
\[
 1\leq x_1 < x_1+x_2 < \cdots < x_1 + x_2 +\cdots + x_{k-1} < x_1+x_2 + \cdots + x_k =n.
\]
By setting $a_1 := x_1$, $a_2:=x_1+x_2$, and so on up to $a_{k-1}$, the composition $(x_1,\ldots,x_k)$ determines (and is uniquely determined by) a set of integers $(a_1,\ldots,a_{k-1})$ satisfying
\[
 1\leq a_1 < a_2 < \cdots < a_{k-1} < n.
\]
Because the $a_i$ are $k-1$ increasing integers between $1$ and $n-1$, the number of such vectors is $\binom{n-1}{k-1}$.

Having established that $|S|=\binom{n-1}{k-1}$, we wish to show that the set of inventories that do \emph{not} form a $k$-gon has cardinality $k\binom{\lfloor n/2 \rfloor}{k-1}$. To this end, we define sets containing the ``bad'' points in $S$; that is, the inventories that do not describe $k$-gons. For $1 \le i \le k$, let
\[
S_i := \{ (x_1,\ldots,x_k) \in S : x_i \geq n/2 \},
\]
which is the set of all inventories that violate the $k$-gon inequality because the segment length $x_i$ is longer than the sum of the other lengths. There can be at most one such segement in any inventory, so the sets $S_1, S_2, \ldots, S_k$ are pairwise disjoint. By Proposition \ref{prop:k-ilateral inequality}, any bad point is contained in some $S_i$, and by symmetry they each have the same cardinality. Thus, the number of bad inventories is:
\[
 |S_1 \cup S_2 \cup \cdots \cup S_k| = \sum_{i=1}^k |S_i| = k|S_1|. 
\]

It remains to prove that $|S_1| = \binom{\lfloor n/2 \rfloor}{k-1}$, which, again, comes down to counting compositions! Consider a point $(x_1,\ldots,x_k) \in S_1$. Thus $x_1+\cdots + x_k = n$ and $x_1 \geq n/2$. Define 
\[
 x'_1 := x_1 - \lceil n/2 \rceil +1 \geq 1.
\]
Then the vector $(x'_1, x_2, \ldots, x_k)$ has 
\[
 x_1' + x_2 +\cdots + x_k = n- \lceil n/2 \rceil +1 = \lfloor n/2 \rfloor + 1.
\]
Thus, inventories in $S_1$ are in bijection with compositions $(x'_1, x_2, \ldots, x_k)$ of $\lfloor n/2 \rfloor + 1$ into $k$ parts. There are $\binom{ \lfloor n/2 \rfloor }{k-1}$ such compositions, which completes the proof.
\end{proof}

We can observe that as $n \rightarrow \infty$, the probability computed in Theorem~\ref{thm:discrete breaking} in the discrete setting approaches the probability computed in Theorem~\ref{thm:splintered breaking}. We do this by noting that for large $m$ and fixed $j$, we can approximate $\binom{m}{j}$ by the polynomial $\frac{1}{j!} m^j$:
\[
 \binom{m}{j} = \frac{1}{j!}m(m-1)\cdots (m-j+1) = \frac{1}{j!}\left( m^j + (\mbox{smaller powers of $m$})\right).
\]
Thus, for fixed $k$,
\[
 \lim_{n\to \infty} \frac{ \binom{\lfloor n/2 \rfloor}{k-1} }{\binom{n-1}{k-1}} = \lim_{n\to \infty} \frac{ \frac{1}{(k-1)!}\left( (n/2)^{k-1} + (\mbox{smaller powers of $n/2$})\right) }{\frac{1}{(k-1)!}\left( (n-1)^{k-1} + (\mbox{smaller powers of $n-1$})\right)} = \frac{1}{2^{k-1}},
\]
and hence
\[
 1- k\cdot \frac{ \binom{\lfloor n/2 \rfloor}{k-1} }{\binom{n-1}{k-1}} \to 1 - \frac{k}{2^{k-1}},
\]
which recovers Theorem \ref{thm:splintered breaking}.

This answer to the general broken stick problem tells us that the probability of forming a $k$-gon increases to 1 rapidly as $k$ increases. This should match our intution: a stick broken into a billion pieces that form a polygon should be akin to forming a circle from a pile of dust.

\section{Pick-up sticks}

The first author found D'Andrea and G\'omez's result so delightful that, for a period of time, he told it to anyone who would listen. To grab a listener's attention, he would focus on the case $k=4$, where the probability is $1-4/8 = 1/2$. Interestingly, most people found this probability to be extremely counter-intuitive, guessing that the probability would be much higher than $1/2$.

When challenged, the first author wrote a computer simulation to convince a skeptic. Disturbingly, the simulation recorded a success rate of $83\%$! As it turned out, rather than coding the ``broken brick problem,'' he had instead implemented a simulation of the following problem, stated here in both continuous and discrete versions.

\begin{quote}
\textbf{The pick-up sticks problem.} Let $k\ge 3$ be a positive integer. Select $k$ sticks of lengths chosen from a uniform distribution of stick lengths. What is the probability that the resulting $k$ sticks form a $k$-gon?
\end{quote}

\begin{quote}
\textbf{The pick-up bricks problem.} Let $n\ge 1$, $k \ge 3$ be positive integers. Select $k$ sticks, each of which has length chosen from the uniform distribution on $\{1,2,\ldots,n\}$. What is the probability that the resulting $k$ sticks form a $k$-gon?
\end{quote}

What the first author had naively assumed was that 
\begin{quote}
``pick up four random sticks'' 
\end{quote}
was basically the same as 
\begin{quote}
``break a random stick into four random pieces,''
\end{quote}
but they are not the same at all! Indeed, the solution to the pick-up bricks problem is as follows.

\begin{theorem}\label{thm:pickup}
Let $n\ge 1$, $k \ge 3$ be positive integers. Pick $k$ distinct sticks, each of which has length chosen from the uniform distribution on $\{1,2,\ldots,n\}$. The probability that the resulting $k$ sticks form a $k$-gon is
$$1 - \frac{k\binom{n+1}{k}}{n^k}.$$
\end{theorem} 

\begin{proof}
The sample space for the pick-up bricks problem is
\[
[1,n]^k = \{ (x_1,\ldots,x_k) : x_i \in \mathbb{Z} \mbox{ and } 1 \leq x_i \leq n \mbox{ for all } i \},
\]
which clearly has cardinality $n^k$. As in the proof of Theorem~\ref{thm:discrete breaking}, we will count the points in the sample space that do \emph{not} form a $k$-gon, showing that there are $k\binom{n+1}{k}$ such ``bad'' points. 

As in the earlier proof, denote the (disjoint) sets of bad points by
\[
S_i := \{ (x_1,\ldots,x_k) \in S : x_i \geq (x_1+\cdots+x_k)-x_i \}.
\]
Again, the total number of bad inventories is:
\[
 |S_1 \cup S_2 \cup \cdots \cup S_k| = \sum_{i=1}^k |S_i| = k|S_1|,
\]
and it remains to prove that $|S_1| = \binom{n+1}{k}$.

Let $(x_1,\ldots,x_k) \in S_1$, meaning that $n \geq x_1 \geq x_2 + \cdots +x_k$. To count such points, we make the following change of variables:
\begin{align*}
 y_1 &:= x_1 +1, \\
 y_2 &:= x_2 + x_3 + x_4 + \cdots + x_k,\\ 
 y_3 &:= x_3 + x_4 + \cdots + x_k,\\
  &\vdots\\
 y_k &= x_k.
\end{align*}
That is, $y_1= x_1+1$ and $y_j = \sum_{i=j}^k x_i$ for $j\geq 2$. Since each $x_i$ is positive, this gives a bijection between the points $(x_1,\ldots,x_k) \in S_1$ and the set of integer points $(y_1,\ldots,y_k)$ satisfying:
\[
n+1\geq y_1 > y_2 > \cdots > y_k \geq 1.
\]
That is, the $y_i$ are simply $k$ distinct numbers between $1$ and $n+1$, and thus there are $\binom{n+1}{k}$ such sets $\{y_1,\ldots, y_k\}$. This shows that $|S_1| = \binom{n+1}{k}$, and the theorem follows.
\end{proof}

We can obtain the solution to the continuous pick-up sticks problem by considering the limit as $n\to \infty$ of the previous result. Using the same estimate for binomial coefficients as before, we have
\[
 \lim_{n\to \infty} \frac{k\binom{n+1}{k}}{n^k} = \frac{k}{k!} \cdot \lim_{n\to \infty} \frac{(n+1)^k + (\mbox{smaller powers of $n+1$})}{n^k} = \frac{1}{(k-1)!},
\]
which gives us the following result.

\begin{corollary}\label{cor:pickup}
Let $k\ge 3$ be a positive integer. Select $k$ sticks of lengths chosen from a uniform distribution of stick lengths. The probability that the resulting $k$ sticks form a $k$-gon is
\[
 1-\frac{1}{(k-1)!}.
\]
\end{corollary}

In light of this result, the $83\%$ coming from the computer simulation of the pick-up sticks problem with $k=4$ was approximating $1-1/3!=5/6$. (Whew!)

Given the elegance of the answer appearing in Corollary \ref{cor:pickup}, the reader may wonder if there is a geometric proof of the pick-up sticks problem, and indeed there is. Since each stick length is chosen independently, the sample space is now a cube. The volume of a region in the cube for which one of the sticks violates the $k$-gon inequality, $x_i \geq \total{S} - x_i$, can be shown to be $1/k!$ of the volume of the whole cube. As there are $k$ such regions and these regions are disjoint, the complementary volume is 
\[
 1-k\cdot\frac{1}{k!} = 1-\frac{1}{(k-1)!}.
\]
See Figure \ref{fig:cube} for an illustration of the case $k=3$.

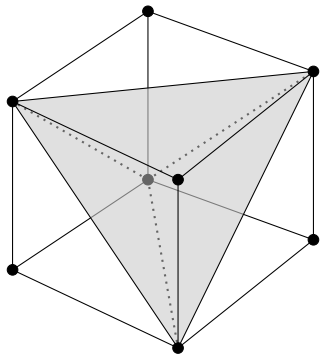
\begin{figure}
 \[
 \begin{tikzpicture}
   \coordinate (o) at (0-.2,0);
   \coordinate (a) at (-2,-1-.2);
   \coordinate (b) at (2,-1+.2);
   \coordinate (f) at (0+.2,-2.24);
   \coordinate (c) at (0-.2,2.24);
   \coordinate (d) at (2,-1+2.24+.2);
   \coordinate (e) at (-2,-1+2.24-.2);
   \coordinate (g) at (0+.2,0);
   \draw (o)--(a)--(e)--(c)--(o)--(b)--(d)--(c);
   \draw (f)--(b);
   \draw[draw=none,fill=black!20,opacity=.2] (o)--(e)--(f)--(o);
   \draw[draw=none,fill=black!20,opacity=.2] (o)--(e)--(d)--(o);
   \draw[draw=none,fill=black!20,opacity=.2] (o)--(f)--(d)--(o);
   \draw[thick,dotted] (o)--(f);
   \draw[thick,dotted] (o)--(e);
   \draw[thick,dotted] (o)--(d);
   \draw[fill=black] (o) circle (2pt);
   \draw[draw=none,fill=black!20,opacity=.5] (f)--(g)--(d)--(f);
   \draw[draw=none,fill=black!20,opacity=.5] (f)--(g)--(e)--(f);
   \draw[draw=none,fill=black!20,opacity=.5] (g)--(e)--(d)--(g);
   \draw (f)--(d)--(e)--(f);
    \draw (g)--(d);
    \draw (g)--(e);
    \draw (g)--(f)--(a);
   \draw[fill=black] (a) circle (2pt);
   \draw[fill=black] (b) circle (2pt);
   \draw[fill=black] (c) circle (2pt);
   \draw[fill=black] (d) circle (2pt);
   \draw[fill=black] (e) circle (2pt);
   \draw[fill=black] (f) circle (2pt);
   \draw[fill=black] (g) circle (2pt);
 \end{tikzpicture}
\]
\caption{The geometric argument for the pick-up sticks problem with $k=3$ sticks. The three corners missing from the cube are each $1/6$ of the volume of the entire cube. Thus the shaded region's volume is $1/2$ of the volume of the whole cube.}\label{fig:cube}
\end{figure}

\section{Four-gon conclusions}

The mistake made when attempting to run a computer simulation of the broken brick problem was a happy accident because it led to another interesting, related question with a satisfying answer. The authors were then motivated to consider a whole host of related questions about combinations of picking up and breaking sticks. 

Consider five different ``four-gon'' problems, each of which can be described in continuous and discrete settings.
\begin{itemize}
\item \textsf{Stick(4)}: one stick breaks into four pieces
\item \textsf{Stick(3,1)}: two sticks of random lengths, one of which breaks into three pieces
\item \textsf{Stick(2,2)}: two sticks of random lengths, each of which breaks into two pieces
\item \textsf{Stick(2,1,1)}: three sticks of random lengths, one of which breaks into two pieces
\item \textsf{Stick(1,1,1,1)}: four sticks of random lengths
\end{itemize}

The classical broken stick/brick problem is the scenario described by \textsf{Stick(4)}. The pick-up sticks/bricks problem is \textsf{Stick(1,1,1,1)}. What about the other three? Computer experiments suggest probabilities of forming a four-gon are $\approx37\%$, $\approx50\%$, and $\approx61\%$ in these scenarios, respectively. We invite interested readers to study these problems and find their own \emph{four-gon conclusions.}

As a broader line of future inquiry, we remark that our labeling of the problems here generalizes to any integer partition. That is, for any integer partition $\lambda_1 \geq \lambda_2 \geq \cdots \geq \lambda_m \geq 1$ with $\sum \lambda_i = k$, we have

\begin{quote}
\textbf{The broken pick-up sticks problem} \textsf{Stick}($\lambda$)\textbf{.} Consider a partition $\lambda=(\lambda_1,\ldots,\lambda_m)$ with $\lambda_1 \geq \lambda_2 \geq \cdots \geq \lambda_m \geq 1$ with $\sum \lambda_i = k$. Pick up $m$ sticks chosen from a uniform distribution of stick lengths. For each $i$, break the $i$th stick into $\lambda_i$ pieces by choosing $\lambda_i-1$ cut points independently at random. What is the probability that the resulting $k$ pieces form a $k$-gon?
\end{quote}

In this article we have given answers to the problems corresponding to the extreme partitions $(k)$ and $(1,1,\ldots,1)$. We encourage the reader to explain other interesting cases, and perhaps solve the problem in complete generality!

\end{document}